\documentclass[12pt, reqno]{amsart}
\usepackage{amsmath, amsthm, amscd, amsfonts, amssymb, graphicx, color}
\usepackage[bookmarksnumbered, colorlinks, plainpages]{hyperref}
\hypersetup{colorlinks=true,linkcolor=red, anchorcolor=green, citecolor=cyan, urlcolor=red, filecolor=magenta, pdftoolbar=true}

\newtheorem{theorem}{Theorem}[section]
\newtheorem{lemma}[theorem]{Lemma}
\newtheorem{proposition}[theorem]{Proposition}
\newtheorem{assumption}[theorem]{Assumption}
\newtheorem{corollary}[theorem]{Corollary}

\newtheorem{question}[theorem]{Question}

\newcommand{\quotes}[1]{``#1''}

\theoremstyle{definition}

\newtheorem{example}[theorem]{Example}

\theoremstyle{remark}
\newtheorem{remark}[theorem]{Remark}

\numberwithin{equation}{section}

\begin{document}

\setcounter{page}{1}

\title[Cauchy transform and polynomial modules]{Cauchy transform and uniform approximation by polynomial modules}

\author[L. Yang]{Liming Yang$^1$}

\address{$^1$Department of Mathematics, Virginia Polytechnic Institute and State University, Blacksburg, VA 24061.}
\email{\textcolor[rgb]{0.00,0.00,0.84}{yliming@vt.edu}}






\begin{abstract}
For a compact subset $K$ of the complex plane $\mathbb C,$ 
let $C(K)$ denote the algebra of continuous functions on $K$.
For an open subset $U \subset K,$
let $A(K,U) \subset C(K)$ be the algebra of functions that are analytic in 
$U.$ We show that there exists $\phi\in A(K,U)$ so that each $f\in A(K,U)$ can uniformly 
be approximated by $\{p_n + q_n\phi\}$ on $K$, where $p_n$ and $q_n$ are 
analytic polynomials in $z$. In particular, $\phi$ can be chosen as a Cauchy 
transform of a finite positive measure $\eta$ compactly supported in $\mathbb C 
\setminus U.$ Recent developments of analytic capacity and Cauchy transform provide 
us useful tools in our proofs.
\end{abstract} \maketitle

\section{\textbf{Introduction}}
\smallskip

Let $\mathcal{P}$ denote the set of polynomials in the complex variable $z.$ 
For a compact subset $K$ of the complex plane $\mathbb C,$ let $Rat(K)$ be the set of all rational functions with poles off $K$ and let $C(K)$ denote the Banach algebra of complex-valued continuous functions on $K$ with customary norm $\|.\|_K$ ($\|.\|_{C(K)}$, or $\|.\|$). Let $P(K)$ and $R(K)$ denote the closures in $C(K)$ of $\mathcal{P}$ and $Rat(K),$ respectively. For an open subset $U 
\subset K,$ let $A(K,U) \subset C(K)$ be the algebra of functions that are analytic in $U.$ We denote  $A(K)=A(K,\text{int}(K)),$ where $\text{int}(K)$ stands for the interior of $K.$ 
For $\phi\in C(K),$ let $P(\phi,K)$ and $PR(\phi,K)$ denote the closure in $C(K)$ of $\mathcal{P} + \mathcal{P} \phi$ and $Rat(K) + \mathcal{P} \phi,$ respectively.

We define the analytic capacity of a compact subset $E$ by
\begin{eqnarray}\label{GammaDefinition}
\ \gamma(E) = \sup |f'(\infty)|,
\end{eqnarray}
where the supremum is taken over all those functions $f$ that are analytic in $\mathbb C_{\infty} \setminus E$ ($\mathbb C_{\infty} = \mathbb C \cup \{\infty \}$), such that
$|f(z)| \le 1$ for all $z \in \mathbb{C}_\infty \setminus E$; and
$f'(\infty) := \lim _{z \rightarrow \infty} z(f(z) - f(\infty)).$
The analytic capacity of a subset $F$ of $\mathbb{C}$ is given by: 
 \[
 \ \gamma (F) = \sup \{\gamma (E) : E\subset F \text{ compact}\}.
 \]
Good sources for basic information about analytic
capacity are \cite{Du10}, Chapter VIII of \cite{gamelin}, \cite{Ga72}, Chapter V of \cite{C91}, and \cite{Tol14}.

The continuous analytic capacity of a compact set $E \subset \mathbb C$ is defined as
 \begin{eqnarray}\label{AlphaDefinition}
 \ \alpha (E) = sup |f' (\infty )|
 \end{eqnarray}
where the supremum is taken over all complex-valued functions which are continuous in $\mathbb C$, analytic on $\mathbb{C}_\infty \setminus E$, and satisfy $|f(z)| \le 1$ for all $z \in \mathbb C$. For a set $F$, we set 
 \[
 \ \alpha (F) = \sup\{\alpha (E) :~ E \subset F;~ E \text{ compact}\}.
 \]
\smallskip

Let $\mathcal L^2$ denote the Lebesgue measure on $\mathbb C$. Let $M_0(\mathbb C)$ 
be the set of finite complex-valued Borel measures that are compactly supported in $
\mathbb C$ and let $M_0^+(\mathbb C)$ be the set of positive measures in 
$M_0(\mathbb C).$
For $\nu\in M_0(\mathbb C)$ and $\epsilon > 0,$ $\mathcal C_\epsilon(\nu)$ is defined by
\ \begin{eqnarray}\label{CTEDefinition}
\ \mathcal C_\epsilon\nu(z) = \int _{|w-z| > \epsilon}\dfrac{1}{w - z} d\nu (w).
\ \end{eqnarray} 
The (principal value) Cauchy transform
of $\nu$ is defined by
\ \begin{eqnarray}\label{CTDefinition}
\ \mathcal C\nu(z) = \lim_{\epsilon \rightarrow 0} \mathcal C_\epsilon\nu (z)
\ \end{eqnarray}
for all $z\in\mathbb{C}$ for which the limit exists. From 
Corollary \ref{CTExist} (1), we see that \eqref{CTDefinition} is defined for all $z$ except for a set of zero analytic 
capacity. Throughout this paper, the Cauchy transform of a measure always means the principal value of the transform. We say $\mathcal C\nu(z)$ is continuous on $\mathbb C$ if $\mathcal C\nu(z)$ coincides $\mathcal L^2$ with a continuous function on $\mathbb C$. We denote $\text{spt}(\nu)$ the support of $\nu.$

The inner boundary of $K$, denoted by $\partial _I K$, is the set of boundary points which do not belong to the boundary of any connected component of $\mathbb C \setminus K$. The inner boundary conjecture (see \cite{VM84}, Conjecture 2) is: if $\alpha(\partial _I K) = 0$, then $R(K) = A(K)$. X. Tolsa \cite{Tol04} shows
that $\alpha$ is semiadditive and affirmatively answers the conjecture. 
For a compact subset $K$ with $\alpha (\partial _I K) > 0$, it is natural to see how \quotes{big} $A(K)$ is compared to $R(K)$. In \cite{y19}, the author constructed a compact subset $K$ and a function $\phi\in A(K)$ such that $R(K) \ne A(K)$ and $A(K) = PR(\phi,K).$ The following question  is also asked in the paper (Problem 3):
\smallskip

\begin{question}\label{MQuestion} 
For each compact subset $K$ of $\mathbb C$, is there a function $\phi\in A(K)$ such that $A(K) = PR(\phi,K)?$
\end{question}
\smallskip

In this paper, our main theorem below affirmatively answers Question \ref{MQuestion} 
as a special case.
 
\begin{theorem}\label{MTheoremIntro2}
If $K \subset \mathbb C$ is a compact subset and $U\subset K$ is an open subset, then there exists 
$\eta\in M_0^+(\mathbb C)$ with $\text{\text{\text{spt}}}(\eta)\subset \mathbb C \setminus U$ such that 
$\mathcal C\eta$ is continuous on $\mathbb C$ (therefore, $\mathcal C\eta\in A(K,U)$) and $A(K,U) = P(\mathcal C\eta, K)$. 
\end{theorem}
\smallskip

Notice that $K$ may not contain $\text{spt}(\eta).$ Otherwise, we obtain the following theorem.

\begin{theorem}\label{MTheoremIntro3}
If $K \subset \mathbb C$ is a compact subset and $U\subset K$ is an open subset, then there exists 
$\eta\in M_0^+(\mathbb C)$ with $\text{\text{\text{spt}}}(\eta)\subset K \setminus U$ such that 
$\mathcal C\eta$ is continuous on $\mathbb C$ (therefore, $\mathcal C\eta\in A(K,U)$) and $A(K,U) = PR(\mathcal C\eta,K).$ 
\end{theorem}
\smallskip

Our proofs rely on remarkable results on (continuous) analytic capacity from \cite{Tol03} and \cite{Tol04} and modified Vitushkin approximation scheme by P. V. Paramonov \cite{p95}. In section 2, we review some recent results of (continuous) analytic capacity and Cauchy transform that are needed in our analysis. 
In section 3, for $\eta\in M_0^+(\mathbb C)$ with $\text{spt}(\eta)\subset \mathbb C \setminus U$ such that $\mathcal C\eta\in A(K,U),$ we introduce a subspace $L(\mathcal C\eta, K) \subset A(K,U),$ the closure of 
$\{\mathcal C(\varphi \eta) :~\varphi \text{ is smooth with compact support}\}$ in $C(K).$ Theorem \ref{MLemma2}, whose proof depends on Paramonov modified Vitushkin approximation scheme, provides a characterization of $L(\mathcal C\eta, K).$   
Interestingly, as a result, we show that $\mathcal C(f\eta)\in L(\mathcal C\eta, K)$ if $f\in L^\infty (\eta)$ and $\mathcal C(f\eta)$ is continuous.
Our main theorem (Theorem \ref{MTheoremIntro1}) in the section proves that $L(\mathcal C\eta, K) + R(K)$ is dense in $A(K,U)$ under certain capacity assumption (see \eqref{MTheoremIntro1Eq}).
Section 4 discusses the assumptions under which $A(K,U) = P(\mathcal C\eta, K).$ It is shown in  Lemma \ref{MLemma1} that if $\eta \in M_0^+(\mathbb C),$ $\mathcal C\eta \in A(K,U),$ and for $\nu \perp P(\mathcal C\eta, K),$
\begin{eqnarray}\label{IIdentityIntro}
 \ \mathcal C\eta(z) \mathcal C\nu(z) = \mathcal C((\mathcal C\eta)\nu)(z), 
 \end{eqnarray}
$\mathcal L^2-a.a.,$ then $L(\mathcal C\eta, K) \subset P(\mathcal C\eta, K).$ Therefore, we investigate the set $A_\eta = \{\lambda:~ \frac{\mathcal C\eta(z) - \mathcal C\eta(\lambda)}{z-\lambda}\in P(\mathcal C\eta, K)\}.$ Clearly, \eqref{IIdentityIntro} holds $\gamma†|_{A_\eta}-a.a..$ Our assumptions (A), (B),and (C) ensure $\mathcal L^2(\mathbb C\setminus A_\eta) = 0$ (see Lemma \ref{GenLemma3}). Together with \eqref{MTheoremIntro1Eq}, we prove that Theorem \ref{MTheoremIntro1} implies $A(K,U) = P(\mathcal C\eta, K)$ (see Theorem \ref{MTheorem}). 
We construct such a measure $\eta$ satisfying the assumptions (A), (B),  (C), and \eqref{MTheoremIntro1Eq} in section 5. Therefore, Theorem \ref{MTheoremIntro2} and Theorem \ref{MTheoremIntro3} are proved.

Before closing this section, we mention some previous related research results. O’Farrell \cite{O75} derives interesting results which relate approximation
problems for the module $R(K) + \sum_{n=1} ^N R(K)\bar z^n$ in different Lipshitz norms and in the uniform norm.
T. Trent and J. Wang \cite{tw81} shows if $K$ is a compact subset without interior, then $R(K) + R(K)\bar z$ is dense in $C(K)$. J. Carmona \cite{C82} generalizes above result to the module $R(K) + R(K)g$ for a smooth function $g$. J. Verdera \cite{v93} proves that each Dini-continuous function in $\overline{A(K) + A(K)\bar z}$ belongs to $\overline{R(K) + R(K)\bar z}$. Finally, the excellent paper \cite{m04}  proves that $R(K) + R(K)\bar z$ is dense in $A(K) + A(K)\bar z$ for any compact subset $K.$

In \cite{t93}, J. Thomson proves if $R(K) \ne C(K),$ then $R(K) + P(K)\bar z$ is not dense in $C(K).$ The papers \cite{y94} and \cite{y95} study the generalized module $R(K) + P(K)g$ and prove that for a smooth function $g$ with $\bar \partial g \ne 0,$ then $R(K) + P(K)g$ is dense in $A(K) + P(K)g$ if and only if $ A(K) = R(K)$. Moreover, \cite{y19} studies the module $R(K) + \sum_{n=1}^N P(K) \bar z^n$ and shows that $R(K) + \sum_{n=1}^N P(K) \bar z^n$ is dense in $A(K) + \sum_{n=1}^N P(K) \bar z^n$ if and only if $ A(K) = R(K)$. The proofs of above results rely on some refinements of a color scheme developed in \cite{t91}. In \cite{BCY16}, A. Baranov, J. Carmona,  and K. Fedorovskiy consider an interesting analogous problem: find necessary and sufficient conditions so that $P(K) +P(K)\bar z^n$ is dense in $A(K) +A(K)\bar z^n$.

\bigskip

\section{\textbf{Preliminaries}}
\smallskip

It is well known that, for $\nu \in M_0(\mathbb C)$, in the sense of distribution,
 \begin{eqnarray}\label{CTDistributionEq}
 \ \bar \partial \mathcal C(\nu) = - \pi \nu.
 \end{eqnarray} 
The following is a simple relationship of $\mathcal L^2$, $\alpha$, and $\gamma$
 \begin{eqnarray}\label{AAlphaGamma}
 \ \mathcal L^2(E) \le 4\pi\alpha (E)^2, ~ \alpha (E) \le \gamma (E)
 \end{eqnarray}
where $E$ is a bounded measurable subset (see Theorem 3.2 on page 204 of \cite{gamelin}).
The maximal Cauchy transform is defined by
 \[
 \ \mathcal C_*(\nu)(z) = \sup _{\epsilon > 0}| \mathcal C_\epsilon(\nu)(z) |.
 \]

A related capacity, $\gamma _+,$ is defined for subsets $E$ of $\mathbb{C}$ by:
\[
\ \gamma_+(E) = \sup \|\eta \|,
\]
where the supremum is taken over $\eta \in M_0^+(\mathbb C)$ with $\text{spt}(\eta) \subset E$ for which $\|\mathcal{C}(\eta) \|_{L^\infty (\mathbb{C})} \le 1.$ 
Since $\mathcal C\eta$ is analytic in $\mathbb{C}_\infty \setminus \mbox{\text{spt}}(\eta)$ and $|(\mathcal{C}(\eta)'(\infty)| = \|\eta \|$, 
we have:
$\gamma _+(E) \le \gamma (E)$
for all subsets $E$ of $\mathbb{C}$. 
The capacity $\alpha_+$ of a bounded set $E \subset \mathbb C$ is defined as
 \[
 \ \alpha_+(E) = \sup \|\eta\|,
 \]
where the supremum is taken over  $\eta \in M_0^+(\mathbb C)$ with $\text{spt}(\eta) \subset E$ such that $\mathcal C(\eta)$ is continuous on $\mathbb C$  and $|\mathcal C(\eta)(z)|\le 1$ for $z\in \mathbb C$. We
clearly have 
$\alpha_+(E) \le \alpha(E)$,
 because $|\mathcal C(\eta)'(\infty )| = \|\eta\|$.
\smallskip

X. Tolsa has established the following astounding results.

\begin{theorem} \label{TTolsa} (Tolsa's Theorem)

(1) $\gamma_+$ and $\gamma$ are equivalent. $\alpha_+$ and $\alpha$ are equivalent. 
That is, there is an absolute positive constant $C_T$ such that 
\begin{eqnarray}\label{GammaEq}
\ \gamma (E) \le C_T \gamma_+(E),~ \alpha (E) \le C_T \alpha_+(E).
\end{eqnarray}

(2) Semiadditivity of analytic capacity:
\begin{eqnarray}\label{Semiadditive}
\ \gamma \left (\bigcup_{i = 1}^\infty E_i \right ) \le C_T \sum_{i=1}^\infty \gamma(E_i), ~ \alpha \left (\bigcup_{i = 1}^\infty E_i \right ) \le C_T \sum_{i=1}^\infty \alpha (E_i).
\end{eqnarray}

(3) There is an absolute positive constant $C_T$ such that, for any $a > 0$, we have:  
\[
\ \gamma(\{\mathcal{C}_*(\nu)  \geq a\}) \le \dfrac{C_T}{a} \|\nu \|.
\]  
\end{theorem}

For (1) and (2), see \cite{Tol03} (also see Theorem 6.1 and Corollary 6.3 in \cite{Tol14}) and \cite{Tol04}.
(3) follows from Proposition 2.1 of \cite{Tol02} (also see \cite{Tol14} Proposition 4.16).
\smallskip

\begin{corollary}\label{CTExist}
Suppose that $\nu , \nu_j \in M_0(\mathbb C)$ for $j \ge 1$. The following statements are true.

(1) There exists $Z\subset \mathbb C$ with $\gamma(Z) = 0$ such that $\lim_{\epsilon\rightarrow 0} \mathcal C_\epsilon(\nu)(z)$ exists for $z\in Z^c$.

(2) For $\epsilon > 0$, there exists a Borel subset $F$ such that $\gamma (F^c) < \epsilon$ and for
 \[
 \ \sup_{z \in F} \mathcal C_* (\nu_j)(z) < \infty, ~ j\ge 1.
 \] 
\end{corollary}

\begin{proof} 
(1) follows from combination Theorem 4.14 and Theorem 8.1 in \cite{Tol14} with Theorem \ref{TTolsa} (1) (also see \cite{To98}).

(2) is an application of Theorem \ref{TTolsa} (2) and (3). In fact, let $A_j = \{\mathcal C_*(\nu_j)(z) \le M_j\}$. By Theorem \ref{TTolsa} (3), we can select $M_j>0$ so that $\gamma(A_j^c) < \frac{\epsilon}{2^{j+1}C_T}$. Set $F = \cap_{j=1}^\infty A_j $. Then applying Theorem \ref{TTolsa} (2), we get
\[
 \ \gamma (F^c) \le C_T \sum_{j=1}^\infty \gamma(A_j^c) < \epsilon.
 \]  
\end{proof} 

Given three pairwise different points $x, y, z \in \mathbb C$, let $R(x, y, z)$ is the radius of the circumference passing through $x, y, z$ (with $R(x, y, z) = \infty$ if $x, y, z$ lie on a same line). For a finite positive measure $\eta$, we set
 \[
 \ c^2_\eta(x) = \int\int \dfrac{1}{R(x,y,z)^2} d\eta(y)d\eta(z),~ c^2(\eta ) = \int c^2_\eta (x)d\eta (x).
 \]

For $\nu\in M_0(\mathbb C)$, define
 $\Theta_\nu (\lambda ) := \lim_{\delta\rightarrow 0} \frac{|\nu |(B(\lambda , \delta ))}{\delta}$
if the limit exists.
The measure $\nu\in M_0(\mathbb C)$ with $\text{spt}(\nu) \subset E$ is $c_0$-linear growth if $|\nu|(B(\lambda, \delta)) \le c_0\delta$ for $\lambda\in \mathbb C$. Write $\nu\in \Sigma(E)$ when $c_0=1$. In addition, if 
$\Theta_\nu(\lambda) = 0$ for $\lambda\in \mathbb C$, we say $\nu\in \Sigma_0(E)$. The statement $A \lesssim B$ (resp. $A\gtrsim B$) means: there exists an absolute constant $C>0$ (resp. $c>0$), independent of $A$ and $B,$ such that $A\le CB$ (resp. $A\ge cB$). The following results are used throughout this paper, we list them here as a lemma.

\begin{lemma}\label{lemmaBasic0}
For $\eta \in M_0^+(\mathbb C)$, suppose $\|\mathcal C\eta\|_{L^\infty (\mathbb C)} \le 1$ and $\mathcal C\eta(z)$ is continuous. The following statements are true.

(1)  
 \begin{eqnarray}\label{lemmaBasic0Eq1}
 \ \eta(B(z, \epsilon))\le \epsilon \sup_{w\in B(z, \epsilon)}|\mathcal C\eta(w) - \mathcal C\eta(z)|
 \end{eqnarray}
and
 \begin{eqnarray}\label{lemmaBasic0Eq2}
 \ |\mathcal C_\epsilon (z) - \mathcal C(z)| \lesssim \sup_{w\in B(z, \epsilon)}|\mathcal C\eta(w) - \mathcal C\eta(z)|.
 \end{eqnarray} 

(2) $\eta(B)^\frac 32 \lesssim \|\eta\|^\frac 12\gamma (B)$ for all bounded Borel subsets $B$, consequently, $\eta(B) = 0$ if $\gamma (B) = 0$.

(3) If $\eta (E) > 0$, then there exists a function $f$ with $0 \le f(z) \le 1$ supported on $E$ such that $\int f(z) d \eta \gtrsim \eta (E)$, $\|\mathcal C(f\eta)\| \lesssim \left (\frac{\|\eta\|}{\eta(E)}\right )^{\frac 12 }$, and $\mathcal C(f\eta)(z)$ is continuous.

(4) Let $\nu\in M_0(\mathbb C)$. If $\sup_{z\in \text{spt}(\eta)}\mathcal C_*(\nu) (z) < \infty$, then  
 \[
 \ \int \mathcal C\eta (z) d \nu (z) = - \int \mathcal C\nu (z) d \eta (z).
 \]
 \end{lemma}

\begin{proof}

(1): Since $\int_{B(0, R)}\int \frac{1}{|z-w|}d\mu(z)d\mathcal L^2(w) < \infty$, we conclude that, by Fubini's theorem, 
 \[
 \ \int_{|z-w|=\epsilon}\int \frac{1}{|z-w|}d\mu(z)|dw| < \infty
 \]
for almost all $\epsilon$. Therefore, for such $\epsilon$, we get
 \[
 \ \begin{aligned}
 \ \eta (B(z, \epsilon)) = &- \dfrac{1}{2\pi i}\int_{|w-z|  =  \epsilon} \mathcal C\mu(w)dw  \\
 \ = &- \dfrac{1}{2\pi i}\int_{|w-z|=\epsilon} (\mathcal C\mu(w) - \mathcal C\mu(z))dw.
 \ \end{aligned}
 \]
\eqref{lemmaBasic0Eq1} follows. 

Let $\phi$ be a function on $\mathbb R$ supported on $[0,1]$ with $0 \le \phi(z) \le 2$ and $\int \phi(|z|) d\mathcal L^2(z) = 1$. Let $\phi_\epsilon(z) = \frac{1}{\epsilon^2} \phi (\frac{|z|}{\epsilon})$, $K_\epsilon = - \frac{1}{z} * \phi_\epsilon$, and $\tilde {\mathcal C}_\epsilon \eta = K_\epsilon * \eta$.
It is easy to show that
$K_\epsilon (z) = - \frac{1}{z}$ for $|z| \ge \epsilon$
and
$\|K_\epsilon \|_\infty \lesssim \frac{1}{\epsilon}$.
Hence,
 \[
 \ |\tilde {\mathcal C}_\epsilon \eta (z) - \mathcal C_\epsilon \eta (z)| = \left |\int_{|z-w | \le \epsilon} K_\epsilon(z - w) 
d\eta (w) \right |  \lesssim \dfrac{\eta (B(z, \epsilon))}{\epsilon}
 \]
and from \eqref{lemmaBasic0Eq1}, we have
\[
 \ \begin{aligned}
 \ & |\mathcal C_\epsilon \eta (\lambda) - \mathcal C\eta (\lambda)| \\
 \ \le & |\tilde {\mathcal C}_\epsilon \eta (\lambda) - \mathcal C_\epsilon \eta (\lambda)| + |\tilde {\mathcal C}_\epsilon \eta (\lambda) - \mathcal C\eta (\lambda)| \\
 \ \lesssim & \dfrac{\eta(B(\lambda, \epsilon))}{\epsilon} + \int |\mathcal C\eta (z) - \mathcal C\eta (\lambda )|\phi_\epsilon(\lambda - z)d\mathcal L^2(z).  
 \ \end{aligned}
 \]
\eqref{lemmaBasic0Eq2} follows.

(2): If $\eta (B) > 0$, then, by Proposition 3.3 in \cite{Tol14}, 
 \[
 \ c^2(\eta | _B) \le c^2(\eta ) \lesssim \|\eta\| \lesssim \dfrac{\|\eta\|}{\eta(B)} \eta(B).
 \]
Set $\eta_B = \left (\frac{\eta(B)}{\|\eta\|} \right )^\frac 12 \eta |_B$. Hence, combining Theorem 4.14 in \cite{Tol14} with Theorem \ref{TTolsa} (1), we conclude $\|\eta_B\| \lesssim \gamma(B)$. 

(3) 
Using Proposition 3.3 in \cite{Tol14}, we have $c^2(\eta) \le a\|\eta\|$, where $a>0$ is an absolute constant. Set $A = \{c^2_\eta(x) \ge \frac{2a\|\eta\|}{\eta(E)}\}$, then
 \[
 \ \eta (A) \le \dfrac{\eta(E)}{2a\|\eta\|} \int _A c^2_\eta(x) d\eta(x) \le \dfrac{\eta(E)}{2}. 
 \]
Set $\eta_0 = \eta |_{E \setminus A}$, then $\|\eta_0\| \ge \frac 12 \eta(E)$.  
Hence,
 \[
 \ c^2(\eta _0) \le \int_{E \setminus A} c^2_\eta(x) d\eta (x) \le \dfrac{2a\|\eta\|}{\eta(E)}\|\eta_0\|.
 \]
Let $\mu = \left (\frac{\eta(E)}{\|\eta\|}\right )^{\frac 12 }\eta_0$, then $\mu\in\Sigma_0(E)$ and 
\[
 \ c^2(\mu ) = \left (\frac{\eta(E)}{\|\eta\|}\right )^{\frac 32} c^2(\eta _0) \le 2a\|\mu\|.
 \]
Using the same proof as in Lemma 3.2 \cite{Tol04} for $\mu$,
 we conclude that there exists a compact subset $E_0 \subset E\setminus A$ such that $\mu (E_0) \ge \frac 12 \|\mu\|$, $\Theta_{\mu} (\lambda) = 0$ for $\lambda\in E_0$ and $c_{\mu}^2 (z) \lesssim 1$ for $z\in \mathbb C$. Using the same proof as in Lemma 3.4 of \cite{Tol04}, we find a compact subset $F\subset E_0$, $\mu(F) \ge \frac 14 \|\mu\|$,  and a function $f$ supported on $F$ with $0\le f(z) \le 1$, and $\int f(x) d\mu(x) \gtrsim \mu(F)$ such that $\|\mathcal C (f\mu)\| \lesssim 1$ and $\mathcal C (f\mu)$ is continuous.

(4): Clearly,
 \[
 \ \int \mathcal C_\epsilon \eta (z) d \nu (z) = - \int \mathcal C_\epsilon \nu (z) d \eta (z).
 \]
(4) follows from Corollary \ref{CTExist} (1), \eqref{lemmaBasic0Eq2}, (2),  and the Lebesgue dominated convergence theorem.

\end{proof}

\bigskip

\section{\textbf{Continuous Cauchy transform and $L(\mathcal C\eta,K)$}}
\smallskip

From now on, we assume $K\subset \mathbb C$ is a compact subset and $U\subset K$ is an open subset. 
For $\delta > 0,$ set $B(\lambda , \delta ) = \{z: |z - \lambda| < \delta \}$.
For a compactly supported smooth function $\varphi$, the localization operator
$T_\varphi$ is defined by
 \begin{eqnarray}\label{VLODefinition}
 \ (T_\varphi f)(\lambda) = \dfrac{1}{\pi}\int \dfrac{f(z) - f(\lambda)}{z - \lambda} \bar\partial \varphi (z) d\mathcal L^2(z),
 \end{eqnarray}
where $f\in L^\infty_c (\mathbb C)$. It is well known that $T_\varphi R(K) \subset R(K)$ and $T_\varphi A(K,U) \subset A(K,U)$. For a continuous function $f$ on $\mathbb C_\infty,$ let
 \[
 \ \omega (f, \delta) = \sup_{z_1,z_2\in B(\lambda, \delta), \lambda \in \mathbb C}|f(z_1) - f(z_2) |. 
 \]
The following
norm estimation is well known for a continuous function $f$:
\begin{eqnarray}\label{TBounded1}
 \ \| T_\varphi f\| \le  8\omega (f, \delta) \delta \|\bar\partial \varphi\|
 \end{eqnarray}
for a smooth function $\varphi$ supported in $B(\lambda,\delta)$.  
For $\nu\in M_0(\mathbb C)$, we see that
\begin{eqnarray}\label{TIdentity}
 \ T_\varphi(\mathcal C \nu)(z) = \mathcal C (\varphi\nu)(z), ~ \mathcal L^2-a.a. 
 \end{eqnarray}

For $\phi\in C(K)$, let $L(\phi, K)$ be the closure of 
 \[
 \ \{T_\varphi \phi:~\varphi \text{ is a smooth function with compact support} \}
 \]
in $C(K)$. Clearly, $L(\phi, K) \subset A(K,U)$ if $\phi \in A(K,U)$. 

For a finite positive measure $\eta$ with $\text{spt}(\eta) \subset \mathbb C\setminus U$ 
and $\mathcal C\eta \in A(K,U),$ Define
 \begin{eqnarray}\label{AlphaEtaPlusDefinition}
 \ \alpha _{\eta +}(B(\lambda, \delta)) = \sup \|\mu\|,
 \end{eqnarray}
where the supremum is taken over all finite positive measures $\mu$, $\text{spt}(\mu) \subset B(\lambda, \delta),$ $\mu$ is
absolutely continuous with respect to $\eta,$ $\mathcal C\mu$ is continuous on $\mathbb C$, and $\|\mathcal C\mu\|_{\mathbb C} \le 1$. We mainly prove the following theorem in this section.

\begin{theorem}\label{MTheoremIntro1}
Let $\eta\in M_0^+(\mathbb C)$ and $\text{spt}(\eta) \subset \mathbb C \setminus U$ such that $\mathcal C\eta$ is continuous on $\mathbb C$. Suppose for $\lambda \in \mathbb C$ and $0 < \delta < \delta_0$ ($\delta_0$ is fixed), 
\begin{eqnarray}\label{MTheoremIntro1Eq}
\ \alpha (B(\lambda, \delta)\setminus U) \le C(\alpha _{\eta +}(B(\lambda, k\delta)) + \gamma (B(\lambda, k\delta)\setminus K))
\end{eqnarray}
for some constants $C>0$ and $k\ge 1$. Then the following properties hold. 
\newline
(1) $L(\mathcal C\eta, K) + R(K)$ is uniformly dense in $A(K,U)$.
\newline
(2) In addition, if $\text{spt}(\eta)$ meets all connected components of $\mathbb C\setminus K$, then $L(\mathcal C\eta, K) = A(K,U)$.   
\end{theorem}

We process our proof of Theorem \ref{MTheoremIntro1} in the remaining section.
In fact, (1) follows from Corollary \ref{MCorollary2} and Corollary
\ref{MCorollary3}. (2) follows from Proposition \ref{MProposition1}.

\smallskip

For $\delta > 0$, let $\{\varphi_{ij},S_{ij}\}_{-\infty < i,j < \infty}$ be a smooth partition of unity, where the length of the square $S_{ij}$ is $\delta$, the support of $\varphi_{ij}$ is in $2S_{ij}$, $0 \le \varphi_{ij} \le 1$, $s_{ij}$, where $Re(s_{ij}) = (i+\frac 12)\delta$ and $Im(s_{ij}) = (j+\frac 12)\delta$, is the center of $S_{ij}$,
 \begin{eqnarray}\label{PartitionUnity}
 \ \|\bar\partial \varphi_{ij} \| \lesssim \frac{1}{\delta},~ \sum \varphi_{ij} = 1,
 \end{eqnarray}
 and 
 \[
 \ \bigcup_{i,j=-\infty}^\infty S_{ij} = \mathbb C,~ Int(S_{ij})\cap Int(S_{i_1j_1}) = \emptyset
 \]
for $(i,j) \ne (i_1,j_1)$.

We provide the following example to show that it is straightforward to construct a finite positive measure $\eta$ with $\text{spt}(\eta) \subset \mathbb C \setminus U$ satisfying the assumptions of Theorem \ref{MTheoremIntro1}. 

\begin{example}
From Theorem \ref{TTolsa} (1), we find $\eta_{ij}^n \in M_0^+(\mathbb C)$ with 
 \[
 \ \text{spt}(\eta_{ij}^n) \subset S_{ij}\cap (K\setminus U),~ \|\eta_{ij}^n\| \ge c\alpha(S_{ij}\cap (K\setminus U)),~ \|\mathcal C(\eta_{ij}^n)\|_{L^\infty (\mathbb C)} \le 1,
 \]
and $\mathcal C(\eta_{ij}^n)$ is continuous on $\mathbb C$, where $S_{ij}$ is defined as in \eqref{PartitionUnity} with $\delta = \frac{1}{2^n}$. 
Let $M_n$ be the number of squares $S_{ij}$ with $\alpha(S_{ij}\cap (K\setminus U)) > 0$. We define
 \[
 \ \eta_n = \dfrac{1}{M_n}\sum_{\alpha(S_{ij}\cap (K\setminus U)) > 0}\eta_{ij}^n.
 \]
Set $\eta = \sum_{n=1}^\infty \frac{1}{2^n}\eta_n.$ By Theorem \ref{TTolsa} (2),
\[
\begin{aligned}
\ \alpha (S_{ij}\setminus U) \le & C_T(\alpha (S_{ij}\cap (K\setminus U)) + \alpha (S_{ij}\setminus K))\\
\ \le & C_T(\frac 1c \|\eta_{ij}^n\| + \gamma (S_{ij}\setminus K))\\
\ \le & C( \alpha _{\eta+}(B(s_{ij},\frac{\sqrt 2}{2}\delta)) + \gamma (S_{ij}\setminus K)).
\end{aligned}
\]
Then it is easy to verify $\eta$ satisfies \eqref{MTheoremIntro1Eq}. Theorem \ref{MTheoremIntro1} (1) implies that $L(\mathcal C\eta, K) + R(K)$ is uniformly dense in $A(K,U)$.
\end{example}
\smallskip

The assumption \eqref{MTheoremIntro1Eq} suggests if $\text{spt}(\eta) \cap (K\setminus U)$ is \quotes{big} compared to $K\setminus U$, then $L(\mathcal C\eta, K) + R(K)$ is uniformly dense in $A(K,U)$.

Let $\eta \in M_0^+(\mathbb C)$ such that $\mathcal C\eta$ is continuous on $\mathbb C$. Define
\begin{eqnarray}\label{AlphaEtaDefinition}
 \ \alpha_\eta (B(\lambda, \delta)) = \sup \|\varphi \eta\|
 \end{eqnarray}
where the supremum is taken over all complex-valued smooth functions $\varphi$ supported in $B(\lambda, \delta)$ and $|\mathcal C(\varphi \eta)(z)| \le 1$ for all $z \in \mathbb C$. Notice that, by \eqref{TIdentity}, $\mathcal C(\varphi\eta)$ is continuous on $\mathbb C$. 

Let $g$ be an analytic function outside the disc $B(a, \delta)$ satisfying the condition
$g(\infty ) = 0.$ We consider the Laurent expansion of $g$ centered at $a$,
 \[
 \ g(z) = \sum_{m=1}^\infty \dfrac{c_m(g,a)}{(z-a)^m}.
 \]
We define $c_1 (g) = c_1(g, a)$.
$c_1 (g)$ does not depend on the choice of $a$, while $c_2 (g,a)$ depends on $a$. However, if $c_1 (g) = 0$, then $c_2 (g,a)$ does not depend on $a$, in this case, we denote $c_2 (g) = c_2 (g,a)$.
\smallskip

The proof of the following theorem relies on modified Vitushkin approximation scheme by P. V. Paramonov in \cite{p95}.

\begin{theorem}\label{MLemma2}
Let $\eta\in M_0^+(\mathbb C)$ such that $\mathcal C\eta$ is continuous on $\mathbb C$. Then for a continuous  function $F$ on $\mathbb C,$ the following conditions are equivalent 
\newline
(1) $F\in L(\mathcal C\eta, K)$.
\newline
(2) There exists a positive function $\omega_F(\delta)$ such that $\omega_F(\delta) \rightarrow 0$ as $\delta \rightarrow 0$ and for a smooth function $\psi$ supported in $B(\lambda, \delta)$, we have
 \begin{eqnarray}\label{MLemma2Eq1}
 \ \left | \int F(z) \bar \partial \psi (z) d \mathcal L^2 (z) \right | \lesssim  \omega_F(\delta) \delta \|\bar \partial \psi\| \alpha_\eta (B(\lambda, \delta)).  
 \end{eqnarray} 
\end{theorem}
\smallskip

\begin{proof}
(1)$\Rightarrow$(2): For a smooth function $\psi$ supported in $B(\lambda, \delta)$ and a smooth function $\varphi$ with compact support, we have, by \eqref{TBounded1},
 \[
 \ \begin{aligned}
 \ \dfrac{1}{\pi}†\left | \int \mathcal C(\varphi \eta)(z) \bar \partial \psi (z) d \mathcal L^2 (z) \right | = &\|\psi \varphi \eta \| \\
 \ \le & \|\mathcal C(\psi\varphi \eta)\| \alpha_\eta (B(\lambda, \delta)) \\
 \ = & \|T_\psi\mathcal C(\varphi \eta)\| \alpha_\eta (B(\lambda, \delta)) \\
 \ \lesssim & \omega(\mathcal C(\varphi \eta), \delta) \delta \|\bar \partial \psi\| \alpha_\eta (B(\lambda, \delta)).
 \ \end{aligned}
 \]
Hence, \eqref{MLemma2Eq1} follows for $\omega_F(\delta) = \omega(F, \delta)$.

 (2)$\Rightarrow$(1):
The modified Vitushkin approximation scheme by P. V. Paramonov in \cite{p95} can be applied in our case. We list the key steps in \cite{p95} below.

Let $f_{ij} = T_{\varphi_{ij}}F$, then, by \eqref{TBounded1}, we have 
\[
\|f_{ij}\| \lesssim \omega (F, \delta)
 \]
and
\begin{eqnarray}\label{MLemma2Eq3}
 \ F = \sum_{ij} f_{ij} = \sum_{2S_{ij} \cap spt(\eta) \ne \emptyset} f_{ij}.
 \end{eqnarray}
For $2S_{ij} \cap \text{spt}(\eta) \ne \emptyset$, by \eqref{MLemma2Eq1}, (2.7) in \cite{p95} becomes
 \begin{eqnarray}\label{MLemma2Eq4}
 \ |c_1 (f_{ij})| = \dfrac{1}{\pi}\left | \int F(z) \bar \partial \varphi _{ij}(z) d\mathcal L^2 (z) \right | \lesssim \omega_F(\delta) \alpha_\eta (B(s_{ij}, \sqrt{2}\delta)).
 \end{eqnarray}
Applying \eqref{MLemma2Eq1} for $\psi = (z - s_{ij})\varphi _{ij}(z)$, (2.8) in \cite{p95} becomes
 \begin{eqnarray}\label{MLemma2Eq5}
 \ \begin{aligned}
 \ |c_2 (f_{ij}, s_{ij})|\lesssim & \omega_F(\delta) \delta \|(z - s_{ij}) \bar \partial \varphi _{ij}(z)\|\alpha_\eta (B(s_{ij}, \sqrt{2}\delta)) \\
 \ \lesssim & \omega_F(\delta) \delta \alpha_\eta (B(s_{ij}, \sqrt{2}\delta)).
 \ \end{aligned}
 \end{eqnarray}

Set $\alpha_{ij} = \alpha_\eta (B(s_{ij}, \sqrt{2}\delta))$. For a smooth function $\varphi$ with compact support in $B(s_{ij}, \sqrt{2}\delta)$ and $g = \mathcal C(\varphi\eta)\in L(\mathcal C\eta, K)$, we have (resp. (2.10) in \cite{p95})
 \[
 \ |c_1(g)| \le \|g\| \alpha_{ij} \lesssim \|g\|\delta
 \]
and
 \begin{eqnarray}\label{MLemma2Eq6}
 \ \begin{aligned}
 \ |c_2(g, s_{ij})| = & |c_1(\mathcal C((z - s_{ij})\varphi \eta))| \\
 \ \le &\|\mathcal C((z - s_{ij})\varphi\eta)\| \alpha_{ij} \\
 \ \lesssim & (\delta \|\mathcal C(\varphi\eta)\| +  |c_1(g)| )\alpha_{ij} \\
 \ \lesssim & \delta \|g\|\alpha_{ij}.
 \ \end{aligned}
 \end{eqnarray} 
 By the definition of $\alpha_\eta$, we find $f_{ij}^* = \frac{c_1 (f_{ij})}{c_1 (\mathcal C(\phi_{ij}\eta))}\mathcal C(\phi_{ij}\eta)$, where $\phi_{ij}$ is a smooth function supported in $B(s_{ij}, \sqrt{2}\delta)$, $\|\mathcal C(\phi_{ij}\eta)\| \le 1$, and $\|\phi_{ij}\eta\| \ge \frac 12 \alpha_{ij}$. Then $c_1(f_{ij}^*) = c_1(f_{ij})$. Set $g_{ij} = f_{ij} - f_{ij}^*$, by \eqref{MLemma2Eq6}, we get
 \[
 \ \|g_{ij}\| \lesssim \omega(\delta), ~ c_1(g_{ij}) = 0, ~~ |c_2(g_{ij}, s_{ij})| \lesssim \omega(\delta)\delta \alpha_{ij},    
 \]
where $\omega(\delta) = \max (\omega_F(\delta), \omega(F,\delta))$. Therefore, (2.16) in \cite{p95} holds. That is, for $|z - s_{ij}| > 3k_1\delta$ ($k_1 \ge 3$ is a fixed integer),
\[
 \ \left | g_{ij}(z)\right | \lesssim \omega(\delta)\left ( \dfrac{\delta \alpha _{ij}}{|z - s_{ij}|^2} + \dfrac{\delta^3}{|z - s_{ij}|^3} \right ).
 \]

The standard Vitushkin approximation scheme requires to that $f_{ij} - f^*_{ij}$ has triple zeros at $\infty$, which requires to estimate both $c_1 (f_{ij})$ and $c_2 (f_{ij}, s_{ij})$. The main idea of P. V. Paramonov is that one does not actually need to estimate
each coefficient $c_2 (f_{ij}, s_{ij})$.  It suffices to do away (with appropriate estimates) with the sum of coefficients $\sum_{j\in I_{is}} c_2(f_{ij}, s_{ij})$ for a special partition $\{I_{is}\}$ into non-intersecting groups $I_{is}$.   

Set $min_i = \min\{j:~ 2S_{ij} \cap \text{spt}(\eta) \ne \emptyset \}$ and $max_i = \max\{j:~ 2S_{ij} \cap \text{spt}(\eta) \ne \emptyset \}$. Let $I_i = \{min_i, min_i+1,..., max_i\}$. 
A concept of a complete group of indices $I \subset I_i$ is defined as in Definition 2.6 in \cite{p95}. Basically, $I$ consists of two subgroups $I_1$ and $I_2$ satisfying $\sum_{j\in I_1}\alpha_{ij} \approx \delta$, $\sum_{j\in I_2}\alpha_{ij} \approx \delta$, and $dist(I_1, I_2)\approx \delta$.

In \cite{p95}, $I_i$ is partitioned into complete groups $I_{i1},...,I_{il_i-1}$ (this family can even be empty) and an incomplete group $I_{il_i} = I_i\setminus(I_{i1}\cup...\cup I_{il_i-1})$ (clearly, there is at most one incomplete group for each $i$).

Let $I=I_{il}$ be a group, define
 \[
 \ g_I = \sum_{j\in I} g_{ij}, ~ c_1(g_I) = \sum_{j\in I} c_1(g_{ij}),~c_2 (g_I) = \sum_{j\in I} c_2 (g_{ij},s_{ij}).
 \]
Then $c_1(g_I) = 0$. For a complete group $I$, Lemma 2.7 in \cite{p95} provides key estimates for $c_2 (g_I)$. In our case, we find $h_{ij} = \frac{\mathcal C(\phi_{ij}\eta)}{\|\mathcal C(\phi_{ij}\eta)\|}$ ($h_j$ in the proof of Lemma 2.7 in \cite{p95}) and $h_I$ ($h_\Gamma$ in Lemma 2.7 of \cite{p95}) as certain linear combinations of $h_{ij}$ such that
 \[
 \  \|h_I\| \lesssim \omega(\delta),~ c_1(h_I) = 0, ~ c_2(h_I) = c_2(g_I).
 \]

We rewrite \eqref{MLemma2Eq3} as the following
 \[
 \ F = \sum_{i} \sum_{l = 1}^{l_i-1} (g_{I_{il}} - h_{I_{il}}) + \sum_{i} \sum_{j\in I_{il_i}} g_{ij} + f_{\delta}
 \]
where
 \[
 \ f_{\delta} = \sum_{i} \sum_{l = 1}^{l_i} \sum_{j\in I_{il}} f_{ij}^* + \sum_{i} \sum_{l = 1}^{l_i-1} h_{I_{il}} \in L(\mathcal C\eta, K).
 \]

The following is proved in \cite{p95}:
 \begin{eqnarray}\label{P95Estimate}
 \ \sum_{i} \sum_{l = 1}^{l_i-1} |g_{I_{il}} - h_{I_{il}}| + \sum_{i} |g_{I_{il_i}}| \lesssim \omega(\delta) .
 \end{eqnarray}
Therefore, $f_{\delta}$ uniformly converges to $F$ . This completes the proof.  
\end{proof}
\smallskip

\begin{corollary}\label{MCorollary1}
Let $\eta\in M_0^+(\mathbb C)$ such that $\mathcal C\eta$ is continuous on $\mathbb C$. Suppose that $\mathcal C(f\eta)$ is continuous on $\mathbb C$ for $f\in L^\infty (\eta)$ with $\|f\|_{L^\infty (\eta)} \le 1$. Then $\mathcal C(f\eta)\in L(\mathcal C\eta, K)$. 
\end{corollary}

\begin{proof}
For a smooth function $\psi$ supported in $B(\lambda, \delta)$, from \eqref{CTDistributionEq}, \eqref{TBounded1}, and \eqref{TIdentity}, we have
 \[
 \ \begin{aligned}
 \ \dfrac{1}{\pi}\left | \int \mathcal C(f\eta)(z) \bar \partial \psi (z) d \mathcal L^2 (z) \right | = &\|\psi f \eta \| \\
 \ \le &\|\psi \eta \| \\
 \ \le & \|\mathcal C(\psi\eta)\| \alpha_\eta (B(\lambda, \delta)) \\
 \ = & \|T_\psi\mathcal C(\eta)\| \alpha_\eta (B(\lambda, \delta)) \\
 \ \lesssim & \omega(\mathcal C\eta, \delta) \delta \|\bar \partial \psi\| \alpha_\eta (B(\lambda, \delta)).
 \ \end{aligned}
 \]
Hence, \eqref{MLemma2Eq1} holds and the corollary follows from Theorem \ref{MLemma2}.
\end{proof}
\smallskip

\begin{corollary}\label{MCorollary2}
Let $\eta\in M_0^+(\mathbb C)$ with $\text{spt}(\eta) \subset\mathbb C \setminus U$ 
such that $\mathcal C\eta$ is continuous on $\mathbb C$. If $\lambda\in \mathbb C$ 
and $0 < \delta < \delta_0$ (for some $\delta_0$)
\[
\ \alpha (B(\lambda, \delta)\setminus U) \le  C( \alpha _\eta (B(\lambda, k\delta)) + \gamma (B(\lambda, k\delta)\setminus K))
\]
for some $k\ge 1$, where $C>0$ is an absolute constant, then $L(\mathcal C\eta, K) + R(K)$ is uniformly dense in $A(K,U)$. 
\end{corollary}

\begin{proof}
Clearly,  $L(\mathcal C\eta, K) \subset A(K,U)$. The same proof of Theorem \ref{MLemma2} applies if we make the following changes. 

If $\frac{1}{2C}\alpha (B(s_{ij}, \sqrt{2}\delta)\setminus U) \le \alpha _\eta (B(s_{ij}, k\sqrt{2}\delta))$, then we set 
$f_{ij}^* = \frac{c_1 (f_{ij})}{c_1 (\mathcal C(\phi_{ij}\eta))}\mathcal C(\phi_{ij}\eta)$ and $h_{ij} = \frac{\mathcal C(\phi_{ij}\eta)}{c_1 (\mathcal C(\phi_{ij}\eta))}$, where $\phi_{ij}$ is a smooth function supported in $B(s_{ij}, k\sqrt{2}\delta)$, $\mathcal C(\phi_{ij}\eta)$ is continuous, $\|\mathcal C(\phi_{ij}\eta)\| \le 1$, and $\|\phi_{ij}\eta\| \ge \frac 12\alpha _\eta (B(s_{ij}, k\sqrt{2}\delta))$.

Otherwise, we have $\frac{1}{2C} \alpha (B(s_{ij}, \sqrt{2}\delta)\setminus U) \le \gamma (B(s_{ij}, k\sqrt{2}\delta)\setminus K)$. Then we set 
$f_{ij}^* = \frac{c_1 (f_{ij})}{c_1 (a_{ij})}a_{ij}$ and $h_{ij} = \frac{a_{ij}}{c_1 (a_{ij})}$, where $a_{ij}$ is analytic off a compact subset of  $B(s_{ij}, k\sqrt{2}\delta) \setminus K$, $\|a_{ij}\| \le 1$, and $a_{ij}'(\infty) \ge \frac 12 \gamma (B(s_{ij}, k\sqrt{2}\delta)\setminus gK)$. 
\end{proof}
\smallskip

It is clear that $\alpha _{\eta}(B(\lambda, \delta)) \le \alpha _{\eta +}(B(\lambda, \delta))$. Next we prove that actually $\alpha _{\eta +}(B(\lambda, \delta)) \approx \alpha _{\eta}(B(\lambda, \delta))$.

\begin{lemma}\label{MTheoremLemma1}
Let $\eta \in M_0^+(\mathbb C)$ such that $\mathcal C\eta$ is continuous on $\mathbb C$. Let $\nu\in L(\mathcal C\eta, K)^\perp \subset C(K)^*$. Then 
 \[
 \ \mathcal C \nu (z) = 0, ~ \eta-a.a. 
 \]
\end{lemma}

\begin{proof}
Suppose that there exists a compact subset $D$ with $\eta(D) > 0$ such that   
 \[
 \ Re(\mathcal C \nu (z)) > 0, ~z\in D.
 \]
By Corollary \ref{CTExist} (2) and Lemma \ref{lemmaBasic0} (2), we may assume that $\mathcal C_* \nu (z)\in L^\infty(\eta |_D)$. 
Using Lemma \ref{lemmaBasic0} (3), we can find a function $w$ supported on $D$ and $0 \le w(x) \le 1$ such that $\int w(x) d \eta(x) \gtrsim \eta (D)$ and $\mathcal C(w\eta)(z)$ is continuous. From Corollary \ref{MCorollary1}, we see that $\mathcal C(w\eta) \in L(\mathcal C\eta, K)$. Using Lemma \ref{lemmaBasic0} (4),  we get
 \[
 \ \int Re(\mathcal C \nu (z)) w(z) d\eta (z)  = - Re \left (\int \mathcal C(w\eta)(z) d \nu (z)\right ) = 0  
 \]
which implies that $Re(\mathcal C \nu (z)) w(z) = 0, ~ \eta-a.a.$. This is a contradiction.  
\end{proof}
\smallskip

\begin{lemma}\label{MTheoremLemma2}
Let $\mu\in M_0^+(\mathbb C)$  such that $\mathcal C\mu$ is continuous on $\mathbb C$ and $\|\mathcal C\mu\| \le 1$. Let $\{\nu_n\}_{n=1}^\infty\subset M_0(\mathbb C)$ such that $\mathcal C\nu_n (z) = 0, ~\mu -a.a.$ Then there exists an absolute constant $C>0$ and a Borel function $w$ with $w \ge 0$ satisfying
 \begin{eqnarray}\label{MTheoremLemma2Eq1}
 \ \|\mu\| \le C \|w\mu\|,~ \|\mathcal C(w\mu)\|_{L^\infty(\mathbb C)}\le C,
  \end{eqnarray}
	$\mathcal C(w\mu)(z)$ is continuous, and 
$\int \mathcal C(w\mu) (z) d \nu_n (z) = 0$ for $n \ge 1$.
\end{lemma}

\begin{proof}
By Corollary \ref{CTExist} (2), we find a subset $E_0 \subset \text{spt}(\mu)$ such that $\gamma( \text{spt}(\mu) \setminus E_0)$ is small and
$\mathcal C_*(\nu_n)(z) \le M_n < \infty,~ z \in E_0$ for $n \ge 1$. 
Using Lemma \ref{lemmaBasic0} (2), we see that $\mu (E_0^c)$ is small and hence, we assume that $\|\mu\| \le 2 \mu(E_0)$. Using Lemma \ref{lemmaBasic0} (3), we find $w$ with $w \ge 0$ and $\text{spt}(w) \subset  E_0$ such that
 \eqref{MTheoremLemma2Eq1} holds and  
$\mathcal C(w\mu)(z)$ is continuous. Applying Lemma \ref{lemmaBasic0} (4), we have
 \[
 \ \int \mathcal C(w\mu)(z) d \nu_n(z)  = - \int \mathcal C\nu_n(z) w(z) d \mu (z) = 0.
 \]
for $n\ge 1$. The lemma is proved. 
\end{proof}
\smallskip

\begin{corollary}\label{MCorollary3}
Let $\eta\in M_0^+(\mathbb C)$ such that $\mathcal C\eta$ is continuous on $\mathbb C$. Then there exists an absolute constant $C>0$ such that $\alpha _{\eta +}(B(\lambda, \delta)) \le C \alpha _{\eta}(B(\lambda, \delta))$.  
\end{corollary}

\begin{proof}
There exists $\mu \in M_0^+(\mathbb C)$ with $\text{spt} (\mu) \subset B(\lambda, \delta)$ such that $\mu$ is
absolutely continuous with respect to $\eta,$ $\mathcal C\mu$ is continuous on $\mathbb C$, $|\mathcal C\mu (z)|\le 1$ for $z\in \mathbb C$, and $ \|\mu\| \ge \frac 12 \alpha _{\eta +}(B(\lambda, \delta))$. Let $\{\nu_n \}\subset L(\mathcal C\eta, K)^\perp$ be a weak-star dense subset. Then by Lemma \ref{MTheoremLemma1}, we have $\mathcal C\nu_n (z) = 0, ~\mu-a.a.$ From Lemma \ref{MTheoremLemma2}, we find $w$ with $w \ge 0$ such that \eqref{MTheoremLemma2Eq1} holds,   
$\mathcal C(w\mu)(z)$ is continuous, and $\int \mathcal C(w\mu) (z) d \nu_n (z) = 0$ for $n \ge 1$. Thus, by Hahn–Banach theorem, we get $\mathcal C(w\mu)\in L(\mathcal C\eta, K)$.  Therefore, $\|w\mu\| \lesssim \alpha _{\eta}(B(\lambda, \delta))$. This completes the proof.    
\end{proof}
\smallskip

Let 
 \begin{eqnarray}\label{ComponentDef}
 \ \mathbb C \setminus K = \bigcup_{m=0}^\infty U_m, ~ U = \bigcup_{m=1}^\infty W_m,
 \end{eqnarray}
where $U_0$ is the unbounded connected component of $\mathbb C \setminus K$, $U_m$ is a bounded connected component  of $\mathbb C \setminus K$ for $m \ge 1$, and $W_m$ is a connected component of $U$.

\begin{proposition}\label{MProposition1}
Let $\eta \in M_0^+(\mathbb C)$ with $\text{spt}(\eta)\subset \mathbb C \setminus U$ such that $\mathcal C\eta$ is continuous on $\mathbb C$. If $\text{spt}(\eta)\cap U_m \ne \emptyset$ for $m \ge 0$ and $L(\mathcal C\eta, K) + R(K)$ is uniformly dense in $A(K,U)$, then $L(\mathcal C\eta, K) = A(K,U)$.  
\end{proposition}

\begin{proof}
Let $\nu \perp L(\mathcal C\eta, K)$. Then $\mathcal C\nu$ is analytic on $U_m$ and by Lemma \ref{MTheoremLemma1}, we get $\mathcal C\nu(z) = 0, ~\eta-a.a.$ Since $\text{spt}(\eta)\cap U_m \ne \emptyset$ and $\eta$ has no atoms (see \eqref{lemmaBasic0Eq1}), for $\lambda\in \text{spt}(\eta)\cap U_m$, we conclude that there exists a sequence $\{\lambda_n\}\subset \text{spt}(\eta)\cap U_m$ such that $\lambda_n\rightarrow \lambda$ and $\mathcal C\nu (\lambda_n ) = 0$. Hence, $\mathcal C\nu (z) = 0$ for $z\in U_m$. This implies $\nu \perp R(K)$. Thus, $R(K) \subset L(\mathcal C\eta, K)$.  
\end{proof} 

\bigskip

\section{\textbf{The space $P(\mathcal C\eta, K)$}}
\smallskip

Let $\nu \perp A(K,U)$. For a bounded Borel function $\psi$ compactly supported in $\mathbb C \setminus U$, we get $\mathcal C(\psi\mathcal L^2) \in A(K,U)$ and
 \[
 \ \int \mathcal C\nu \psi d \mathcal L^2 = - \int \mathcal C(\psi\mathcal L^2) d\nu = 0 
 \]
which implies $\mathcal C\nu(z) = 0,~\mathcal L^2 |_{\mathbb C\setminus U}-a.a.$ 
For $f\in A(K,U)$ and $\lambda \in U$, it is clear that $\frac{f(z) - f(\lambda)}{z-\lambda}\in A(K,U)$. Therefore,
 \begin{eqnarray}\label{IIdentity}
 \ f(z) \mathcal C\nu(z) = \mathcal C(f\nu)(z),~\mathcal L^2-a.a. 
 \end{eqnarray}
The identity \eqref{IIdentity} is an important and useful property. In this section, we discuss some properties related to \eqref{IIdentity} for a closed subspace $\mathcal M\subset A(K,U)$, which will be used in next section for our construction of $\eta$ to ensure $A(K,U) = P(\mathcal C\eta, K)$. We start with the following simple lemma.  

\begin{lemma}\label{MLemma1}
Let $\mathcal M$ be a closed subspace of $A(K,U)$. Let $\eta \in M_0^+(\mathbb C)$ such that $\mathcal C\eta$ is continuous. If for $\nu \perp \mathcal M$,
\begin{eqnarray}\label{IIdentity1}
 \ \mathcal C\eta(z) \mathcal C\nu(z) = \mathcal C((\mathcal C\eta)\nu)(z),~\mathcal L^2-a.a. 
 \end{eqnarray}
then $L(\mathcal C\eta, K) \subset \mathcal M$.
\end{lemma}

\begin{proof}
Let $\varphi$ be a smooth function with compact support.
We have the following calculation
 \[
 \ \begin{aligned}
 \ & \int T_\varphi \mathcal C\eta(z) d \nu(z) \\
 \ = & \dfrac{1}{\pi}\int \int \dfrac{\mathcal C\eta(z) - \mathcal C\eta(w)}{z - w} \bar \partial \varphi (w) d\mathcal L^2(w) d \nu(z) \\
 \ = & \dfrac{1}{\pi}\int (\mathcal C ((\mathcal C\eta)\nu) (w) - \mathcal C\eta (w) \mathcal C \nu (w)) \bar \partial \varphi (w) d\mathcal L^2(w)\\
 \ = & 0. 
 \ \end{aligned}
 \]
The lemma is proved.
\end{proof}
\smallskip 

Define 
 \[
 \ A_{\eta} = \{\lambda \in \mathbb C:~ \mathcal C \eta (z) - \mathcal C \eta (\lambda) = (z - \lambda)F_\lambda,\text{ for some }F_\lambda \in P(\mathcal C\eta, K)\}.
 \]
For $\nu \perp P(\mathcal C\eta, K)$, we get
 \begin{eqnarray}\label{GenEq50}
 \ \mathcal C\eta(z)\mathcal C(\nu)(z) = \mathcal C((\mathcal C\eta)\nu)(z), ~ \gamma |_{A_{\eta}}-a.a..
 \end{eqnarray}

We need some technical assumptions on $A_\eta$ below. 

\begin{assumption} (A): $\gamma (A_{\eta}\cap W_m) > 0$ for $m\ge 1$.

(B): For $m\ge 1$, $U_m \cap \text{spt}(\eta) \ne \emptyset$, $\text{int} (U_m \cap \text{spt}(\eta))= \emptyset$, $U_m \setminus \text{spt}(\eta)$ is a non-empty connected open subset, and $(U_m \setminus \text{spt}(\eta)) \cap A_{\eta}$ contains a sequence and its limit.

(C): There exists $B\subset K\setminus U$ with $\mathcal L^2(B) = 0$ such that for $\lambda \in (K\setminus U) \setminus B$, we have
 \[
 \ \underset{\delta\rightarrow 0}{\overline{\lim}}\dfrac{\gamma (B(\lambda, \delta) \cap (A_\eta\cup\bigcup_{m=1}^\infty U_m))}{\delta} > 0.
 \]
\end{assumption}
\smallskip

By \eqref{ComponentDef}, $\mathbb C = (\cup_{m=1}^\infty W_m)\cup(\cup_{m=0}^\infty U_m)\cup(K\setminus U).$ If $W_m\cap A_\eta,$ $U_m\cap A_\eta,$ and $(K\setminus U)\cap A_\eta$ are some \quotes{small} non-empty subsets as in the above assumptions, using a continuity property of Cauchy transform (see Lemma \ref{CauchyTLemma}), we will be able to prove \eqref{GenEq50} actually holds $\mathcal L^2-a.a.$ for $z\in \mathbb C$ (see Lemma \ref{GenLemma3}).

We will construct a finite positive measure $\eta$ that satisfies the assumptions (A), (B), and (C) in next section. The following lemma is due to Lemma 3.2 in \cite{acy19}.

\begin{lemma}\label{CauchyTLemma} 
Let $\nu\in M_0(\mathbb C)$. For some $\lambda _0$ in $\mathbb C$, if $\Theta_{\nu} (\lambda _0 ) = 0$ and $\mathcal{C} (\nu)(\lambda _0) = \lim_{\epsilon \rightarrow 0}\mathcal{C} _{\epsilon}(\nu)(\lambda _0)$ exists. then:

(1) $\mathcal{C}(\nu)(\lambda) = \lim_{\epsilon \rightarrow 0}\mathcal{C} _{\epsilon}(\nu)(\lambda )$ 
exists for $\lambda\in Z^c$ with $\gamma(Z) = 0$ and

(2) for $a > 0$,  
 \[ 
 \ \lim_{\delta\rightarrow 0}\dfrac{\gamma (B(\lambda_0, \delta) \cap \{|\mathcal{C}(\nu)(\lambda) - \mathcal{C}(\nu)(\lambda_0)| > a\})}{\delta} = 0.
 \]
\end{lemma}
\smallskip

\begin{lemma}\label{GenLemma1}
Let $\eta \in M_0^+(\mathbb C)$ with $\text{spt}(\eta)\subset\mathbb C\setminus (U_0 \cup U)$ such that $\mathcal C \eta$ is continuous on $\mathbb C$. Let $\nu \perp P(\mathcal C\eta, K)$. The two properties below hold.
\newline
(1) If assumption (A) holds, then
 \begin{eqnarray}\label{GenEq51}
 \ \mathcal C\eta(z)\mathcal C(\nu)(z) = \mathcal C((\mathcal C\eta)\nu)(z), ~ \mathcal L^2|_{W_m}-a.a.,~ m \ge 1.
 \end{eqnarray}
\newline
(2) If assumption (B) holds, then
\begin{eqnarray}\label{GenEq5}
 \ \mathcal C(\nu)(z) = \mathcal C((\mathcal C\eta)\nu)(z) = 0, ~ z\in U_m,~ m \ge 0,
 \end{eqnarray}
 and
 \begin{eqnarray}\label{GenEq52}
 \ \mathcal C(\nu)(z) = \mathcal C((\mathcal C\eta)\nu)(z) = 0, ~  ~ \mathcal L^2 |_{\partial U_m}-a.a.,~ m \ge 0.
 \end{eqnarray}

\end{lemma}

\begin{proof}
(1) Since $\text{spt}\eta \cap W_m = \emptyset$, $\mathcal C\eta(z)$ is analytic on $W_m$. From \eqref{CTDistributionEq}, we have
 \[
 \ \bar \partial (\mathcal C\eta(z)\mathcal C(\nu)(z) - \mathcal C((\mathcal C\eta)\nu)(z)) = 0,~ \mathcal L^2|_{W_m}-a.a.
 \]
By Weyl’s lemma, there exists an analytic function $a(z)$ on $W_m$ such that
\[
 \ \mathcal C\eta(z)\mathcal C(\nu)(z) - \mathcal C((\mathcal C\eta)\nu)(z) = a(z),  ~ \mathcal L^2|_{W_m}-a.a.
 \]
By \eqref{GenEq50} and Lemma \ref{CauchyTLemma}, for $\lambda \in A_\eta \cap W_m, ~\gamma-a.a.,$
\[ 
 \ \lim_{\delta\rightarrow 0}\dfrac{\gamma (B(\lambda, \delta) \cap \{|\mathcal C\eta(z)\mathcal C(\nu)(z) - \mathcal C((\mathcal C\eta)\nu)(z)| > \epsilon\})}{\delta} = 0,
 \]
 therefore, from \eqref{AAlphaGamma},
 \[ 
 \ \lim_{\delta\rightarrow 0}\dfrac{\mathcal L^2 (B(\lambda, \delta) \cap \{|a(z)| > \epsilon\})}{\delta} = 0.
 \]
Hence, we see $a(z)=0,  ~ \gamma |_{A_\eta\cap W_m}-a.a.$ \eqref{GenEq51} follows.
  
(2) The case for $m=0$ follows from $\text{spt}(\eta)\cap U_0 = \emptyset$ and $P(K) \subset P(\mathcal C\eta, K)$. Since $\mathcal C(\nu)(z)$ and $\mathcal C((\mathcal C\eta)\nu)(z)$ are analytic on $U_m$ and $\mathcal C\eta(z)$ is analytic on $U_m\setminus \text{spt}\eta$,  we get from assumption (B) that,
\[
 \ \mathcal C\eta(z)\mathcal C(\nu)(z) = \mathcal C((\mathcal C\eta)\nu)(z), ~ z\in U_m\setminus \text{spt}\eta.
 \]
From the fact that  $\mathcal C\eta(z)$ is continuous on $\mathbb C$ and assumption (B), we conclude that
 \[
 \ \mathcal C\eta(z)\mathcal C(\nu)(z) = \mathcal C(\mathcal C\eta\nu)(z), ~ \mathcal L^2|_{ U_m}-a.a.
 \]
Using \eqref{CTDistributionEq} and taking $\bar\partial$ both sides, we get $\mathcal C(\nu)(z) = 0,~\eta|_{U_m}-a.a.$ Hence, 
\[
 \ \mathcal C(\nu)(z) = \mathcal C(\mathcal C\eta\nu)(z) = 0, ~ z\in U_m.
 \]

There is a subset $Z$ with $\mathcal L^2(Z) = 0$ such that for $\lambda_0\in \partial U_m\setminus Z$, we have
 $\int \frac{1}{|z-\lambda_0|} d |\nu| < \infty$.
Applying Lemma \ref{CauchyTLemma}, we have
 \[
 \ \lim_{\delta\rightarrow 0}\dfrac{\gamma(B(\lambda_0,\delta)\cap\{|\mathcal C\nu(\lambda) - \mathcal C\nu(\lambda_0)|>\epsilon\})}{\delta} = 0.
 \]
Let $P_{\lambda_0, \delta}$ be a path stating at a point in  $B(\lambda_0, \frac{\delta}{2})$ and ending at a point in $U_m \setminus B(\lambda_0, \delta)$ such that $P_{\lambda_0, \delta}\subset U_m$. Using Theorem \ref{TTolsa} (2), we get 
\[
 \ \begin{aligned}
 \ &\underset{\delta\rightarrow 0}{\overline{\lim}} \dfrac{\gamma(B(\lambda_0, \delta)\cap P_{\lambda_0, \delta} \cap \{|\mathcal C\nu(\lambda) - \mathcal C\nu(\lambda_0)|\le\epsilon\})}{\delta} \\
 \ \ge &\dfrac{1}{C_T}\underset{\delta\rightarrow 0}{\overline{\lim}} \dfrac{\gamma(B(\lambda_0, \delta)\cap P_{\lambda_0, \delta})}{\delta} - \underset{\delta\rightarrow 0}{\lim} \dfrac{\gamma(B(\lambda_0, \delta)\cap \{|\mathcal C\nu(\lambda) - \mathcal C\nu(\lambda_0)| > \epsilon\})}{\delta} \\
 \ = &\dfrac{1}{C_T}\underset{\delta\rightarrow 0}{\overline{\lim}} \dfrac{\gamma(B(\lambda_0, \delta)\cap P_{\lambda_0, \delta})}{\delta} \\
 \ > &0.
 \ \end{aligned}
 \]
There exists a sequence $\lambda_k\in P_{\lambda_0, \delta_k} \subset U_m$ with $\lambda_k\rightarrow \lambda_0$ and $\mathcal C\nu(\lambda_k)\rightarrow \mathcal C\nu(\lambda_0)$, which implies $\mathcal C\nu(\lambda_0) = 0$ by \eqref{GenEq5}. Similarly, $\mathcal C(\mathcal C\eta\nu)(\lambda_0) = 0$. \eqref{GenEq52} is proved.
\end{proof}
\smallskip

\begin{lemma}\label{GenLemma3}
Let $\eta  \in M_0^+(\mathbb C)$ with $\text{spt}(\eta) \subset\mathbb C\setminus (U_0 \cup U)$ such that $\mathcal C \eta$ is continuous on $\mathbb C$. If $A_\eta$ satisfies assumption (A), (B), and (C), then for $\nu \perp P(\mathcal C\eta, K)$,
 \begin{eqnarray}\label{GenEq6}
 \ \mathcal C\eta(z)\mathcal C(\nu)(z) = \mathcal C((\mathcal C\eta)\nu)(z), ~ \mathcal L^2-a.a.
 \end{eqnarray}
\end{lemma}

\begin{proof} From Lemma \ref{GenLemma1}, it remains to prove \eqref{GenEq6} for $\lambda \in K\setminus U$. 
Now for $\lambda \in K\setminus U$ with $\int \frac{1}{|z-\lambda|} d |\nu |(z) < \infty$, we set $B(\nu, \lambda, \epsilon) = \{|\mathcal C\nu(z) - \mathcal C\nu(\lambda)|>\epsilon\}$. Using Lemma \ref{CauchyTLemma}, we get
 \[
 \ \underset{\delta\rightarrow 0}{\lim} \dfrac{\gamma(B(\lambda, \delta)\cap B(\nu, \lambda, \epsilon))}{\delta} = \underset{\delta\rightarrow 0}{\lim} \dfrac{\gamma(B(\lambda, \delta)\cap B((\mathcal C\eta)\nu, \lambda, \epsilon))}{\delta} = 0.
 \]
By Theorem \ref{TTolsa} (2), \eqref{GenEq5},  and assumption (C), we see that
 \[
 \ \begin{aligned}
 \ &\underset{\delta\rightarrow 0}{\overline{\lim}} \dfrac{\gamma(B(\lambda, \delta)\cap B(\nu, \lambda, \epsilon)^c \cap B((\mathcal C\eta)\nu, \lambda, \epsilon)^c \cap\{\mathcal C\eta(z)\mathcal C(\nu)(z) = \mathcal C((\mathcal C\eta)\nu)(z)\})}{\delta} \\
 \ \ge &\dfrac{1}{C_T}\underset{\delta\rightarrow 0}{\overline{\lim}} \dfrac{\gamma(B(\lambda, \delta)\cap\{\mathcal C\eta(z)\mathcal C(\nu)(z) = \mathcal C((\mathcal C\eta)\nu)(z)\})}{\delta} \\
 \ & - \underset{\delta\rightarrow 0}{\lim} \dfrac{\gamma(B(\lambda, \delta)\cap B(\nu, \lambda, \epsilon))}{\delta} - \underset{\delta\rightarrow 0}{\lim} \dfrac{\gamma(B(\lambda, \delta)\cap B((\mathcal C\eta)\nu, \lambda, \epsilon))}{\delta} \\
 \ \ge &\dfrac{1}{C_T}\underset{\delta\rightarrow 0}{\overline{\lim}} \dfrac{\gamma(B(\lambda, \delta)\cap (A_\eta\cup \bigcup_{m=1}^\infty U_m))}{\delta} \\
 \ > &0.
 \ \end{aligned}
 \]
Therefore, there exists $\{\lambda_k\}$ such that $\lambda_k \rightarrow \lambda$, $\mathcal C\eta(\lambda_k)\mathcal C(\nu)(\lambda_k) = \mathcal C((\mathcal C\eta)\nu)(\lambda_k)$, $\mathcal C\eta(\lambda_k)\rightarrow \mathcal C\eta(\lambda)$, $\mathcal C(\nu)(\lambda_k)\rightarrow \mathcal C(\nu)(\lambda)$, and $\mathcal C((\mathcal C\eta)\nu)(\lambda_k)\rightarrow \mathcal C((\mathcal C\eta)\nu)(\lambda)$. Hence, $\mathcal C\eta(\lambda)\mathcal C(\nu)(\lambda) = \mathcal C((\mathcal C\eta)\nu)(\lambda)$. The lemma is proved.            
\end{proof}
\smallskip

Assuming there exists a finite positive measure $\eta$ satisfying assumption (A), (B), (C), and \eqref{MTheoremEq}, we prove below that each $f\in A(K,U)$ can be uniformly approximated by $p_n + q_n \mathcal C\eta$, where $p_n,q_n \in \mathcal P$. The construction of such a measure $\eta$ involves many technicalities, so we leave it in next section.    

\begin{theorem}\label{MTheorem}
Let $\eta  \in M_0^+(\mathbb C)$ with $\text{spt}(\eta) \subset \mathbb C\setminus (U_0 \cup U)$ such that $\mathcal C \eta$ is continuous on $\mathbb C$. If $A_\eta$ satisfies assumptions (A), (B), (C), and 
\begin{eqnarray}\label{MTheoremEq}
\ \alpha (B(\lambda, \delta)\setminus U) \lesssim \alpha _{\eta +}(B(\lambda, k\delta)) + \gamma (B(\lambda, k\delta)\setminus K)
\end{eqnarray}
for $\lambda\in \mathbb C$ and $0 < \delta < \delta_0$ (for some $\delta_0 > 0$),
then $A(K,U) = P(\mathcal C\eta, K)$.
\end{theorem}

\begin{proof}
It follows from Lemma \ref{GenLemma1} (2) that $R(K) \subset P(\mathcal C\eta, K)$. Lemma \ref{MLemma1} and Lemma \ref{GenLemma3} imply $L(\mathcal C\eta, K) \subset P(\mathcal C\eta, K)$. Now the theorem follows from Theorem \ref{MTheoremIntro1}.
\end{proof}
\smallskip

Before closing this section, we briefly discuss the space $PR(\mathcal C\eta, K).$ Define 
 \[
 \ A_{\eta}' = \{\lambda \in \mathbb C:~ \mathcal C \eta (z) - \mathcal C \eta (\lambda) = (z - \lambda)F_\lambda,\text{ for some }F_\lambda \in PR(\mathcal C\eta, K)\}.
 \]

\begin{assumption} (A'): $\gamma (A_{\eta}'\cap W_m) > 0$ for $m\ge 1$.

(B'): For $m\ge 1$, $U_m \cap \text{spt}(\eta) = \emptyset$ and $U_m \cap A_{\eta}'$ contains a sequence and its limit.

(C'): There exists $B\subset K\setminus U$ with $\mathcal L^2(B) = 0$ such that for $\lambda \in (K\setminus U) \setminus B$, we have
 \[
 \ \underset{\delta\rightarrow 0}{\overline{\lim}}\dfrac{\gamma (B(\lambda, \delta) \cap (A_\eta '\cup\bigcup_{m=1}^\infty U_m))}{\delta} > 0.
 \]
\end{assumption}

Notice that if $\nu\perp PR(\mathcal C\eta, K),$ then $\mathcal C\nu (z) = 0$ for $z\in U_m$ and $m \ge 0.$ Therefore, Lemma \ref{GenLemma1} and Lemma \ref{GenLemma3} hold for assumptions (A'), (B'), and (C'). Thus, we obtain the following theorem for $PR(\mathcal C\eta, K).$

\begin{theorem}\label{MTheoremPR}
Let $\eta  \in M_0^+(\mathbb C)$ with $\text{spt}(\eta) \subset K\setminus U$ such that $\mathcal C \eta$ is continuous on $\mathbb C$. If $A_\eta '$ satisfies assumptions (A'), (B'), (C'), and 
\begin{eqnarray}\label{MTheoremEq}
\ \alpha (B(\lambda, \delta)\setminus U) \lesssim \alpha _{\eta +}(B(\lambda, k\delta)) + \gamma (B(\lambda, k\delta)\setminus K)
\end{eqnarray}
for $\lambda\in \mathbb C$ and $0 < \delta < \delta_0$ (for some $\delta_0 > 0$),
then $A(K,U) = PR(\mathcal C\eta, K)$.
\end{theorem}

\bigskip

\section{\textbf{Construction of $\eta$ satisfying the assumptions of Theorem \ref{MTheorem} or Theorem \ref{MTheoremPR}}}
\smallskip

In this section, we construct a measure $\eta$ and prove (by Theorem \ref{MTheorem} or Theorem \ref{MTheoremPR}) the following theorem.

\begin{theorem}\label{MTheorem2}
The following statements are true.

(1) There exists $\eta \in M_0^+(\mathbb C)$ with $\text{spt}(\eta) \subset \mathbb C\setminus (U_0 \cup U)$ such that $\mathcal C \eta$ is continuous on $\mathbb C$ and $A(K,U) = P(\mathcal C\eta, K)$.

(2) There exists $\eta \in M_0^+(\mathbb C)$ with $\text{spt}(\eta) \subset K\setminus U$ such that $\mathcal C \eta$ is continuous on $\mathbb C$ and $A(K,U) = PR(\mathcal C\eta, K)$.

\end{theorem}
\smallskip

Let $S_{ij}^n$ be defined as in \eqref{PartitionUnity} with $\delta = \frac{1}{2^n}$ and center $Re(s_{ij}) = (i+\frac 12)\frac{1}{2^n}$ and $Im(s_{ij}) = (j+\frac 12)\frac{1}{2^n}.$
Let $l_{in} = \{z:~Re(z) = \frac{i}{2^n}\}$ be a vertical line. Define $\{L_m\}$ as the collection of vertical lines as the following:
\[
\ \{L_m\}_{m=1}^\infty = \{l_{in}:~ l_{in} \cap K \ne \emptyset, ~ -\infty < i < \infty, ~ n = 1,2,...\}.
\]
Our approach is to construct such a measure $\eta$ so that $\cup_m L_m \subset A_\eta.$

Let $U_m$ and $W_m$ be as in \eqref{ComponentDef}. For $m\ge 1,$ let $K_m\subset U_m$ be a compact subset with no interior such that $U_m \setminus K_m$ is connected and $\alpha(K_m) > 0$.
Set
\begin{eqnarray}\label{GenEq000}
 \ K_o = (K\setminus U) \cup \bigcup_{m=1}^\infty K_m.
 \end{eqnarray} 
and
 \begin{eqnarray}\label{GenEq00}
 \ E = K_o \setminus \left ( \bigcup_{m=1}^\infty L_m \right ).
 \end{eqnarray} 

 We assume $\alpha( 2S_{ij}^n\cap E) > 0$ for \eqref{GenEq0}, \eqref{GenEq10}, and \eqref{GenEq1} below. By Theorem \ref{TTolsa} (2), we have
 \begin{eqnarray}\label{GenEq0}
 \ \begin{aligned}
 \ &\alpha( 2S_{ij}^n\cap E) \\
 \ \ge &\alpha( (2S_{ij}^n\cap K_o) \setminus \bigcup_{m=1}^\infty (2S_{ij}^n\cap K_o\cap L_m) \\
 \ \ge &\dfrac{1}{C_T}\alpha(2S_{ij}^n\cap K_o) - \sum_{m=1}^\infty \alpha(2S_{ij}^n\cap K_o\cap L_m) \\
 \ = &\dfrac{1}{C_T}\alpha(2S_{ij}^n\cap K_o). 
\ \end{aligned}
 \end{eqnarray}
 where $\alpha(L_m) = 0.$
From Theorem \ref{TTolsa} (1), we find $\eta_{ij}^n\in M_0^+(\mathbb C)$ with 
 \begin{eqnarray}\label{GenEq10}
 \ \text{spt}(\eta_{ij}^n) \subset 2S_{ij}^n\cap E,~\mathcal C(\eta_{ij}^n) \text{ is continuous},
 \end{eqnarray}
and, by \eqref{GenEq0},
 \begin{eqnarray}\label{GenEq1}
 \ \|\eta_{ij}^n\| \gtrsim \alpha( 2S_{ij}^n\cap E)\gtrsim \alpha(2S_{ij}^n\cap K_o),~ \|\mathcal C(\eta_{ij}^n)\|_{L^\infty (\mathbb C)} \le 1.
 \end{eqnarray}
Let $M_n$ be the number of squares $S_{ij}^n$ with $\alpha( 2S_{ij}^n\cap E) > 0$. We define
 \[
 \ \eta_n = \dfrac{1}{M_n}\sum_{\alpha(2S_{ij}\cap E) > 0}\eta_{ij}^n.
 \]

The measure $\eta_n$ satisfies the following properties
\newline
(A1) $\mathcal C(\eta_n) \in A(K,U)$;
\newline
(A2) $\|\mathcal C(\eta_n)\|_{L^\infty (\mathbb C)} \le 1$; and 
\newline
(A3) For $m \ge 1$, $\text{spt}(\eta_{ij}^n)  \cap L_m = \emptyset$ and $\frac{\mathcal C(\eta_n)(z) - \mathcal C(\eta_n)(\lambda)}{z - \lambda} \in A(K,U)$ for $\lambda \in L_m$. Therefore,
 \[
 \ B_n := \sup_{\lambda \in \cup_{k=1}^nL_k}\left(\int \dfrac{1}{|z-\lambda|} d\eta_n(z) + \left \| \dfrac{\mathcal C(\eta_n)(z) - \mathcal C(\eta_n)(\lambda)}{z - \lambda} \right \|_{L^\infty (\mathbb C)}\right ) < \infty.
 \]  
\smallskip

Let $\{u_k\}$ be a dense subset of $\cup_{m=1}^\infty L_m$ such that $\{u_k\}\cap L_m$ is dense in $L_m$ for $m\ge 1$. Set $f_k(z) = \frac{1}{z-u_k}$.
There is $m_1$ such that $u_1\in L_{m_1}$. Set
 $d_1 = \text{dist}(L_{m_1},\text{spt}(\eta_1))$, then $d_1 > 0$. 
Let $D_1 = \{x:~ \text{dist}(x,\text{spt}(\eta_1)) \le \frac{1}{2}d_1\}$. 
Since $L_{m_1}$ is a subset of the unbounded component of $\mathbb C \setminus D_1$, we can find a polynomial $p_{11}$ such that
 \[
 \ \left \|p_{11}(z) - \dfrac{1}{z-u_1}\right \|_{C (\partial \hat D_1)} = \left \|p_{11}(z) - \dfrac{1}{z-u_1}\right \|_{C (D_1)} \le \dfrac{d_1}{2(4\|\eta_1\| + 1)}
 \]
where $\hat D_1$ denotes the polynomial convex hull of $D_1$.
We have the following calculation (the maximum modulus principle is applied):
 \begin{eqnarray}\label{PolyEstimate}
 \ \begin{aligned}
 \ & \left \|\mathcal C(p_{11}\eta_1)(z) - \dfrac{\mathcal C\eta_1(z) - \mathcal C\eta_1(u_1)}{z - u_1} \right \|_{C (D_1)} \\
 \ \le &\left \|\int \left (\dfrac{p_{11}(w) - p_{11}(z)}{w - z} - \dfrac{f_1(w) - f_1(z)}{w - z}\right )d \eta_1 (w) \right \|_{C (D_1)} \\
 \  & + \|p_{11} - f_1 \|_{C (D_1)} \| \mathcal C\eta_1 \|_{C (D_1)} \\
 \ \le &\int \left \|\dfrac{p_{11}(w) - p_{11}(z)}{w - z} - \dfrac{f_1(w) - f_1(z)}{w - z}\right \|_{C (\partial \hat D_1)}d \eta_1 (w) \\
 \  & + \|p_{11} - f_1 \|_{C (D_1)} \\
 \ \le & \int \dfrac{2\|p_{11} - f_1\|_{C (\partial \hat D_1)}}{\text{dist}(w,\partial \hat D_1)}d \eta_1 (w) + \|p_{11} - f_1 \|_{C (D_1)}  \\
 \ \le & \dfrac 12. 
 \ \end{aligned}
 \end{eqnarray}
Set $a_1 = \min(\frac{1}{2}, \frac{1}{2B_1})$ and $\xi_1 = a_1 \eta_1$. We find $L_{m_2}$ with $u_2\in L_{m_2}$.  
Set
$d_2 = \text{dist}(L_{m_1}\cup L_{m_2},\text{spt}(\eta_1) \cup \text{spt}(\eta_2))$, then  $d_2> 0$.
Let $D_2 = \{x:~ \text{dist}(x,\text{spt}(\eta_1) \cup \text{spt}(\eta_2)) \le \frac{1}{2}d_2\}$.
Let 
 \[
 \ A_2 = \left \|\mathcal C(p_{11}\eta_2)(z) - \dfrac{\mathcal C(\eta_2)(z) - \mathcal C(\eta_2)(u_1)}{z - u_1} \right \|_{C (D_2)}
 \]
and
 \[
 \ a_2 = \min \left (\dfrac{1}{4}, \dfrac{1}{4A_2}, \dfrac{1}{4B_2} \right ), ~ \xi_2 = \xi_1 + a_2\eta_2. 
 \]
Then
 \[
 \ \left \|\mathcal C(p_{11}\xi_2)(z) - \dfrac{\mathcal C(\xi_2)(z) - \mathcal C(\xi_2)(u_1)}{z - u_1} \right \|_{C (D_2)} \le \frac12 + \frac14.
 \]
Similarly to \eqref{PolyEstimate}, we can find two polynomials $p_{k2}$ such that  
 \[
 \ \left \|\mathcal C(p_{k2}\xi_2)(z) - \dfrac{\mathcal C(\xi_2)(z) - \mathcal C(\xi_2)(u_k)}{z - u_k} \right \|_{C (D_2)} \le \frac14
 \]
for $k = 1,2$.

We find $L_{m_k}$ with $u_k\in L_{m_k}$.  
Set
$d_k = \text{dist}(\cup_{i=1}^kL_{m_i},\cup_{i=1}^k\text{spt}(\eta_i))$, then $d_k > 0$.
Let $D_k = \{x:~ \text{dist}(x,\cup_{i=1}^k\text{spt}(\eta_i)) \le \frac{1}{2}d_k\}$.
Therefore, we can find polynomials $\{p_{kj}\}_{k\le j}$ and positive measures $\{\xi_j\}$ such that, for $k \le l \le j + 1$,
 \[
 \ \left \|\mathcal C(p_{kl}\xi_{j+1})(z) - \dfrac{\mathcal C(\xi_{j+1})(z) - \mathcal C(\xi_{j+1})(u_k)}{z - u_k} \right \|_{C (D_{j+1})} \le \sum_{i = l}^{j+1}\frac{1}{2^i} \le \frac{2}{2^l}.
 \]
Using the maximum modulus principle, we have
 \[
 \ \left \|\mathcal C(p_{kl}\xi_{j+1})(z) - \dfrac{\mathcal C(\xi_{j+1})(z) - \mathcal C(\xi_{j+1})(u_k)}{z - u_k} \right \|_{C (K\cup K_o)}  \le \frac{2}{2^l}.
 \]
It is clear that, by the construction, 
 \begin{eqnarray}\label{EtaDefinition}
 \ \eta = \sum_{n=1}^\infty a_n\eta_n
 \end{eqnarray}
 is well defined. We conclude that from (A1)-(A3) and for $u\in L_m$, $\|\mathcal C\eta\|_{L^\infty(\mathbb C)} \le 1$, $\mathcal C\eta$ is continuous, $\frac{\mathcal C\eta(z) - \mathcal C\eta(u)}{z-u} = \mathcal C(\frac{\eta}{w-u})(z)~ \mathcal L^2-a.a.$, $\|\mathcal C(\frac{\eta}{w-u})\|_{L^\infty(\mathbb C)} \le 1$, and $\mathcal C(\frac{\eta}{w-u})(z)$ is continuous. Therefore, for $k \le l$, 
 \[
 \ \left \|\mathcal C(p_{kl}\eta)(z) - \dfrac{\mathcal C\eta(z) - \mathcal C\eta(u_k)}{z - u_k} \right \|_{C(K\cup K_o)} \le \frac{1}{2^{l-1}}.
 \]
Notice that 
 \[
 \ \begin{aligned}
 \ & \mathcal C(p_{kl}\eta)(z) \\
 \ = & \int \dfrac{p_{kl}(w) - p_{kl}(z)}{w - z} d\eta(w) + p_{kl}(z)\mathcal C\eta(z) (\in P(\mathcal C\eta,K)), ~\mathcal L^2-a.a.
 \ \end{aligned} 
 \]
Hence,
 \[
 \ \dfrac{\mathcal C\eta(z) - \mathcal C\eta(u_k)}{z - u_k} \in P(\mathcal C\eta,K).
 \]
By the construction, it is easy to show that for $\lambda \in L_m$, there exists a sequence $\{u_{n_k}\}\subset L_m$ with $u_{n_k}\rightarrow \lambda$ and
 \[
\ \left \| \dfrac{\mathcal C\eta(z) - \mathcal C\eta(u_{n_k})}{z - u_{n_k}} - \dfrac{\mathcal C\eta(z) - \mathcal C\eta(\lambda)}{z - \lambda} \right \|_{C(K\cup K_o)} \rightarrow 0. 
\]
Thus,
 \begin{eqnarray}\label{GenEq4}
\ \dfrac{\mathcal C\eta(z) - \mathcal C\eta(\lambda)}{z - \lambda} \in P(\mathcal C\eta,K) 
\end{eqnarray}
for $\lambda \in L_m$ and $m \ge 1$.
From Corollary \ref{CTExist} (1) and for $\nu \perp P(\mathcal C\eta,K)$, we conclude that 
 \[
 \ \mathcal C\eta(z)\mathcal C(\nu)(z) = \mathcal C((\mathcal C\eta)\nu)(z), ~ \gamma |_{L_m}-a.a.
 \]
\smallskip

Now we are ready to prove that $\eta$ satisfies the assumptions of Theorem \ref{MTheorem}.

\begin{proof} (Theorem \ref{MTheorem2} (1))
By \eqref{GenEq4}, we have
 \[
 \ \bigcup_{m=1}^\infty L_m \subset A_\eta.
 \]

There exists $L_{m_0}$ such that $\gamma (L_{m_0} \cap W_m) > 0.$ Assumption (A) follows.

There exists $L_{m_0}$ such that $\gamma (L_{m_0} \cap (U_m \setminus K_m)) > 0.$ Assumption (B) follows.

Assumption (C) holds since for $\lambda \in K\setminus U$, there is a subsequence $\{L_{m_k}\}$ such that $dist(\lambda, L_{m_k}) \rightarrow 0.$ Therefore,
 \[
 \ \underset{\delta\rightarrow 0}{\overline{\lim}} \dfrac{\gamma(B(\lambda, \delta)\cap \bigcup_{m=1}^\infty L_m)}{\delta} > 0.
 \]
For \eqref{MTheoremEq}, by Theorem \ref{TTolsa} (2) and \eqref{AAlphaGamma}, we have the following calculation.
 \[
 \ \begin{aligned}
 \ & \alpha(2S_{ij}^n\setminus U) \\
 \ \le & C_T(\alpha(2S_{ij}^n\cap (K\setminus U)) + \alpha(2S_{ij}^n\setminus K)) \\
 \ \lesssim & \alpha(2S_{ij}^n\cap  (K\setminus U)) + \gamma(2S_{ij}^n \setminus K).
 \ \end{aligned}
 \]
\eqref{MTheoremEq} follows from
 \[
 \ \alpha(2S_{ij}^n\cap  (K\setminus U)) \lesssim \|\eta_{ij}^n\| \lesssim \alpha _{\eta+} (B(s_{ij}, \sqrt 2 \frac {1}{2^n}))
 \]
 by \eqref{GenEq1}.
The proof of Theorem \ref{MTheorem2} (1) now follows from Theorem \ref{MTheorem}.
   
\end{proof}
\smallskip

For Theorem \ref{MTheorem2} (2), we only  need to change \eqref{GenEq000} to the following:

\begin{eqnarray}
 \ K_o = K\setminus U.
 \end{eqnarray} 
 \smallskip
 
 \begin{remark}
 Note that we can choose $\eta$ such that 
 \[
 \ \eta(\sum_{m=0}^\infty \partial U_m) = 0.
 \]	
 This is ensured by Lemma \ref{GenLemma1} (2).
 \end{remark}

\bigskip

\centerline{\bf Acknowledgment}

The author would like to thank the referee for carefully reading the manuscript and providing helpful comments.

\bigskip

\appendix

\bibliographystyle{amsalpha}

\begin{thebibliography}{A}

\bibitem [ACY19]{acy19}J. R. Akeroyd, J B. Conway, L.Yang \textit{On nontangential limits and shift invariant subspaces}, Integral Equations Operator Theory, {\bf 91} (2019), no. 1, Art. 2, 18 pp.

\bibitem [BCY16]{BCY16} A.D. Baranov, J.J. Carmona, K.Yu. Fedorovskiy, \textit{Density of certain polynomial modules}, J. Approx. Theory {\bf 206} (2016), 1--16.
		
\bibitem [C82]{C82} J. J. Carmona, \textit{A necessary and sufficient condition for uniform approximation by certain
rational modules}, Proc. Amer. Math. Soc., {\bf 86} (1982), no. 3, 487--490.

\bibitem [C91]{C91} J. B. Conway, \textit{The theory of subnormal operators}, Mathematical Survey and Monographs 36, 1991.

\bibitem [Du10]{Du10} J. J. Dudziak, \textit{Vitushkin's conjecture for removable sets}, Universitext, Springer, New York, 2010.

\bibitem [G69]{gamelin} T. W. Gamelin, \textit{Uniform algebras}, American Mathematical Society, Rhode Island, 1969.

\bibitem [Ga72]{Ga72} J. Garnett, \textit{Analytic capacity and measure}, Lecture Notes in Mathematics, Vol 297, Springer-Verlag, Berlin, 1972.

\bibitem [M04]{m04} M.Ya. Mazalov, \textit{Uniform approximations by bianalytic functions on arbitrary compact subsets in $\mathbb C$}, Sb.
Math. {\bf 195}, (2004), no. 5, 687--709.

\bibitem [O75]{O75} A. O'Farrell, \textit{Annihilators of rational modules}, J. Funct. Anal. {\bf 19} (1975), 373--389.

\bibitem [P95]{p95}  P. V. Paramonov, \textit{Some new criteria for uniform approximability of functions by rational fractions} Sbornik: Mathematics {\bf 186} (1995), no. 9, 1325-1340.

\bibitem [T91]{t91} J. E. Thomson, \textit{Approximation in the mean by polynomials}, Ann. of Math.  {\bf 133} (1991), no. 3, 477--507.

\bibitem [T93]{t93} J. E. Thomson, \textit{Uniform approximation by rational functions}, Indiana Univ. Math. J  {\bf 42} (1993), no. 17, 167–-177.


\bibitem [To98]{To98} X. Tolsa, \textit{Cotlar’s inequality without the doubling condition and existence of principal values for the Cauchy integral of measures}, J. Reine Angew. Math.  {\bf 502} (1998), 199--235.

\bibitem [To02]{Tol02} X. Tolsa, \textit{On the analytic capacity $\gamma+$}, Indiana Univ. Math.J.  {\bf 51} (2002), no. 2, 317--343.

\bibitem [To03]{Tol03} X. Tolsa, \textit{Painleves problem and the semiadditivity of analytic capacity}, Acta Math.  {\bf 51} (2003), no. 1, 105--149.

\bibitem [To04]{Tol04} X. Tolsa, \textit{The semiadditivity of continuous analytic capacity and the inner boundary conjecture}, Amer. J. Math., {\bf 126} (2004), no. 3, 523–-567.

\bibitem [To14]{Tol14} X. Tolsa, \textit{Analytic capacity, the Cauchy transform, and non-homogeneous Calderon-Zygmund
theory}, Birkhauser/Springer, Cham, 2014.

\bibitem [TW81]{tw81} T. Trent and J. L. Wang, \textit{Uniform approximation by rational modules on nowhere dense sets}, Proc.
Amer. Math. Soc.,  {\bf 81} (1981), 62–-64.

\bibitem [V93]{v93} J. Verdera. \textit{On the uniform approximation problem for the square of the Cauchy-Riemann operator}
Pacific J. Math., {\bf 159}, (1993), no. 2, 379–-392.

\bibitem [VM84]{VM84} A. G. Vitushkin and M. S. Melnikov, \textit{Analytic capacity and rational approximation}, Linear and complex analysis, Problem book, Lecture Notes in Math. {\bf 1043}, Springer-Verlag, Berlin, 1984.

\bibitem[Y94]{y94} L. Yang, \textit{On uniform approximation problems and T-invariant algebras}, Indiana Univ. Math. J,  {\bf 43} (1994), 639–-650.

\bibitem[Y95]{y95} L. Yang, \textit{Uniform rational approximation}, Proc. Amer. Math. Soc.,  {\bf 123} (1995), no. 1, 201–-206.


\bibitem[Y19]{y19} L. Yang, \textit{Bounded point evaluations for certain polynomial and rational modules}, Journal of Mathematical Analysis and Applications  {\bf 474} (2019), 219--241.

\end{thebibliography}

\end{document}